\documentclass[11pt]{amsart}
\usepackage{graphicx}
\usepackage{amssymb}
\usepackage{epstopdf}

\usepackage{amsmath,amsfonts,amsthm,amsopn,cite,mathrsfs}

 \theoremstyle{theorem}
\newtheorem{lemma}{Lemma}
\newtheorem{theorem}{Theorem}[section]
\newtheorem{corollary}[theorem]{Corollary}

\newtheorem{proposition}[theorem]{Proposition}

\theoremstyle{remark}

\newtheorem*{remark}{Remark}

\newcommand{\CC}{\mathbb{C}\,}
\newcommand{\DD}{\mathbb D}
\newcommand{\TT}{\mathbb T}

\def\mS{{\mathcal S}}
\def\D{{\mathbb D}}
\def\T{{\mathbb T}}

\newcommand{\vf}{\varphi}

\newcommand{\vt}{\vartheta}

\newcommand{\Dd}{D_\delta}
\newcommand{\K}{\mathcal K}

\DeclareMathOperator{\dist}{dist}
\DeclareMathOperator{\diam}{diam}
\DeclareMathOperator{\supp}{supp}
\newcommand{\eps}{\epsilon}

\newcommand{\beq}{\begin{equation}}
 \newcommand{\eeq}{\end{equation}}

\begin{document}

\title[Trace ideal criteria on model spaces]{Trace ideal criteria for embeddings and composition operators on model spaces}
 
\author[A. Aleman,  Yu. Lyubarskii, E. Malinnikova, K.-M. Perfekt]
{Alexandru Aleman, Yurii Lyubarskii, Eugenia Malinnikova, Karl-Mikael Perfekt}

%small changes
\begin{abstract}
Let $K_\vt$ be a model space generated by an inner function $\vt$. We study the Schatten class membership of embeddings $I : K_\vt \hookrightarrow L^2(\mu)$, $\mu$ a positive measure, and of composition operators $C_\vf:K_\vt\rightarrow H^2(\D)$ with a holomorphic function $\vf:\D\rightarrow\D$. In the case of one-component inner functions $\vt$ we show that the problem  can be reduced to the study of natural extensions of $I$ and $C_\vf$ to the Hardy-Smirnov space $E^2(D)$ in some  domain  $D\supset \D$.  In particular, we obtain a characterization of  Schatten membership  of $C_\vf$ in terms of  Nevanlinna counting function.   By example this characterization does not hold true for general $\vt$.            
\end{abstract}
\thanks{This work was carried out at the Center for Advanced Study, Norwegian Academy of Science and Letters. Yu.L. and E.M. are partially supported by project 213638 of the Norwegian Research Council}
\maketitle

\section{Introduction}

Let $\DD=\{z  :  |z|<1\}$ be the unit disk  and $\T=\{z: |z|=1\}$ be the unit circle. A bounded analytic function $\vt$ in $\D$ is said to be {\em inner }
if its non-tangential boundary values satisfy $|\vt|=1$ almost everywhere on $\T$. We denote by $H^2(\D)$ the Hardy space on $\D$ and by  
 $K_\vt= H^2(\DD)\ominus \vt H^2(\DD)$  the corresponding model space. 
 
  In this article  two   classes of operators are considered:  embeddings   $I_\mu : K_\vt \to L^2(\mu)$, where $\mu$  is a  finite positive measure supported on $\overline{\DD}$,
and composition operators 
$
 C_\vf: f \mapsto f\circ \vf
 $ 
 acting from  $K_\vt$ into $H^2(\DD)$, where $\vf: \DD \to \DD$ is a  holomorphic function. In fact, it is well-known that the latter type of operator may be considered as a special case of the former for a certain pullback measure $\mu_\vf$. We mention that embeddings of model spaces have   been studied by a number of authors \cite{C, C2, VT, CM, B}; composition operators on
Hardy (and more general) spaces is by now a classical  subject -- we refer the reader to \cite{Saksman}
for a description of the current state of the art and a history survey. In this article we study the composition operator acting
from the model space $K_\vt$ into $H^2(\DD)$ thus emphasis interaction between the
boundary behavior of $\vf$ and the spectrum of the inner function $\vt$. In such setting the problem has been considered in \cite{LM}. 
% Since we study composition operators acting from the model space   $K_\vt$  into  $H^2(\DD)$, a setting previously considered in \cite{LM}, there is an emphasis on the interaction between the boundary behavior of $\vf$ and the spectrum of the inner function $\vt$.
 Our main goal is to understand when such embedding and composition operators belong to the Schatten trace ideal $\mS_p$, $0 < p < \infty$.

%Most of our results require the hypothesis that $K_\vt$ be generated by a {\em one-component} inner function $\vt$.  Indeed, 
The embedding operators on $K_\vt$ have  proved easier to analyze when $\vt$ is a {\em one-component} inner function, %due to the many additional tools available in this case, 
see \cite{A, VT, CM, B}.
% Most relevant to our objective, we mention that 
In particular, the Schatten ideal membership of $I_\mu$ has   been characterized by Baranov \cite{B} for one-component $\vt$. In Section 3  we suggest a different  approach to the problem. We return to the  the original ideas of Cohn \cite{C} and  extend embedding operators on $K_\vt$ to operators acting on the Hardy-Smirnov space $E^2(D)$ of a certain domain $D\supset \D$. This allows us to obtain a geometrical criterion for the inclusion of $I_\mu$ in $\mS_p$. In particular we  recover the aforementioned result in \cite{B}. 

For composition operators $C_\vf$ we further refine our result to give trace ideal criteria in terms of the Nevanlinna counting function $N_\vf$,
\[
N_\vf(z)=\sum_{\vf(\zeta)=z} \log \frac 1 {|\zeta|}.
\]
 We  combine the geometric  approach   with recent results \cite{LQLR, LQLR2} that clarify the connection between the Nevanlinna counting function $N_\vf$ and the measure $\mu_\vf$,
in order to obtain the following characterization.

{
\renewcommand{\thetheorem}{\ref{thm:integralcond}}
\begin{theorem} 
Let $\vt$ be a one-component inner function. The operator $C_\vf:K_\vt\rightarrow H^2$ is in $\mS_p$, $0 < p < \infty$, if and only if 
\[
\int_\DD\left(\frac{N_\vf(z)(1-|\vt(z)|)^2}{1-|z|^2}\right)^{p/2}\left(\frac{1-|\vt(z)|^2}{1-|z|^2}\right)^2dA<\infty.\]
\end{theorem}
\addtocounter{theorem}{-1}
} 

\medskip
The article is organized  as follows. The next section contains preliminary information about one-component inner functions and the corresponding model spaces. In Section 3 
we reduce the trace ideal problem of embedding operators on $K_\vt$ to a corresponding problem in the Hardy-Smirnov space in a larger domain, leading to a geometric characterization.   In Section 4 
 we use these results in order  to   describe when $C_\vf\in \mS_p$ in terms of the Nevanlinna counting function of $\vf$. We also give some geometric examples, illustrating the Schatten class behavior of composition operators on the Paley-Wiener space.  General model spaces are treated in Section 5, giving an example that the one-component requirement of Theorem \ref{thm:integralcond} can not be dropped, and providing sufficient conditions for $C_\vf$ to belong to $\mS_p$, $p\ge 2$. 

 %%%%%%%%%%%%%%%%%%%
%%%%%%%%%%%%%%%%%%%%%

\section{Preliminaries}
\subsection{One-component inner functions}

  We recall that the class of {\em one-component} inner functions, introduced in \cite{C}, consists of those inner functions $\vt$ such that, for some $q_0\in (0,1)$, the set 
 \[
D_\epsilon=\{z\in\DD: |\vt(z)|<\epsilon\}
\]     
is connected for all $\eps\in (q_0,1)$. 
We refer the reader to   \cite{C} and \cite{C2} for the basic facts  about one-component inner functions.
For the remainder of this section  we fix a one-component function $\vt$ and a corresponding number $q_0$.

 Consider the canonical factorization of $\vt$,     
 \[
\vt(z)=B_\Lambda(z)  \exp \left ( -\int_\TT \frac {\xi+z}{\xi-z} d\omega (\xi)  \right ),
\]
 where $\Lambda$ is the zero set of $\vt$, $B_\Lambda$ is the corresponding Blaschke product, and $\omega$
is a singular measure on $\TT$.  Functions in $K_\vt$ admit analytic continuation through $\TT \setminus  \Sigma(\vt)$, 
where

\[
%\label{spectrum}
\Sigma(\vt) =  \left (\TT\cap {\rm Clos}(\Lambda) \right ) \cup {\rm supp}(\omega)
\]
is the {\em spectrum}  of $\vt$ (see \cite{Nik}, Lecture 3).  The function $\vt$ itself can be reflected over $\TT\setminus\Sigma(\vt)$ and thus definition of the domain $D_\eps$ makes sense for any $\eps\in (q_0,q_0^{-1})$.

 Our construction is based on the following results from \cite{C}.
 
\begin{proposition}
 \begin{itemize}
Let $\delta\in (1, q_0^{-1})$. Then 
\item  The set $\Dd$ is simply connected,  its boundary $\partial \Dd$ is a rectifiable  Jordan curve, and the linear measure on $\partial \Dd$ is a Carleson measure with respect to $\D^-=\{z: |z|>1\}$. 
\item Any function $f\in K_\vt$ extends analytically to a function in $E^2(\Dd)$ and 
\[
\|f\|_{K_\vt}  \simeq \|f\|_{E^2(\Dd)} , \quad  f \in K_\vt.
\]
\end{itemize}
 \end{proposition}
Here and in what follows $E^2(\Dd)$ and  $E^2_-(\Dd)$ denote  the Hardy-Smirnov spaces in the interior and exterior of $\Dd$ (see e.g. \cite{D}). We mention that functions $g \in E^2_-(\Dd)$ are required to assume the value $0$ at $\infty$.

Recall that a rectifiable curve $\Gamma$ is called {\em Ahlfors regular}  if for each $z\in \CC$ and $r>0$ we have 
 $H^1(\Gamma\cap B(z,r))< C r$, where $H^1(\cdot)$ denotes arc length, $B(z,r)=\{\zeta: |\zeta-z|<r\}$, and $C = C(\Gamma) > 0$ is some constant.

\begin{lemma}\label{l:Ahl}
Let $\delta\in (1, q_0^{-1})$. Then the boundary of $D_\delta$ satisfies the Ahlfors regularity condition.
\end{lemma}

\begin{proof}
Denote $q=\delta^{-1}$.
The boundary of $D_\delta$ is a rectifiable Jordan curve that is the reflection of the curve $\partial D_q=\vt^{-1}(|z|=q)\cup \Sigma(\vt) $ with respect to the unit circle, see \cite{C}. 

First, if $B(z,r)$ is a ball centered at a point $z \in \Sigma(\vt)$,  $r > 0$ sufficiently small, there exists a constant $C > 0$ such that the Carleson box $S\left( \tilde{z} \right)$ centered at $\tilde{z} = (1-Cr)z$ contains $B(z,r) \cap \partial D_q$. Since the arc length on $\partial D_q$ is a Carleson measure, we obtain
\begin{equation*}
H^1(B(z,r) \cap \partial D_q) \leq H^1( S(\tilde{z}) \cap \partial D_q) \lesssim 1-|\tilde{z}| \leq Cr,
\end{equation*}
which is precisely the Ahlfors regularity estimate for points $z \in \Sigma(\vt)$.

For points in the level set $\Gamma_q=\vt^{-1}(|z|=q)$, we again want to show the existence of constant $C$, independent of $z$, such that $H^1(\Gamma_q\cap B(z,r))\le Cr$, for all $z\in\Gamma_q$. By the same argument as in the previous paragraph, for a fixed $a > 0$ we have desired estimate for all balls with radius $r>a(1-|z|)$. For comparatively smaller balls, we note that Theorem 1.1 and Lemma 2.1 of \cite{C} imply that there exists $\gamma=\gamma(q)$ such that for any $z\in \Gamma_q$ the restriction $\vt|B(z, \gamma(1-|z|))$ is univalent. The proof is then completed by the Hayman-Wu theorem \cite{HW}, since it implies that $H^1( \vt^{-1}(|z|=q)\cap B(z,r))\lesssim r$ whenever $r<\gamma (1-|z|)$.
\end{proof}

\begin{corollary} \label{cor:proj}
The space $L^2(\partial \Dd)$ admits the direct sum decomposition
\[
L^2(\partial \Dd)=E^2(\Dd)\dotplus E^2_-(\Dd).
\]
The corresponding projectors $P_\pm$ are bounded and have the form
\beq
\label{proj}
P_\pm f(z)= \pm \frac 12 f(z) + \frac 1 {2i\pi }  \int_{\partial \Dd}  
                 \frac {f(\zeta)}{\zeta-z} d\zeta, \quad z \in \partial \Dd.
\eeq
\end{corollary}
The proof is straightforward; it suffices to mention that the Ahlfors regularity yields the boundedness  
of $P_\pm$ in $L^2(\partial \Dd)$ by David's theorem \cite{David}.

 %%%%%%%%%%%%%%%%%%%%%%%%%%%%%

\subsection{Model spaces}
  Each function $f\in K_\vt=H^2(\DD)\cap \vt H_{-}^2(\DD)$ admits an extension to a function in  $E^2(D_\delta)$. Denote by $J$ the operator of 
  analytic prolongation  from $\D$ to $\Dd$,  and let $\K_{{\vt}}=E^2(D_\delta)\cap {\vt}E^2_-(D_\delta)$.

\begin{proposition}
\label{l:kth}
  $J(K_\vt)=\K_{{\vt}} $.
\end{proposition} 
%\noindent {\em Proof}. 
\begin{proof} The inclusion  $J(K_\vt)\subset \K_{{\vt}}$ follows from Cohn's extension construction \cite{C} which shows that $J(K_\vt)\subset E^2(D_\delta)$, and the observation that if $f\in K_\vt$ then $\vt^{-1}f$ has an analytic continuation in $\CC\setminus\DD$.
 
 In order to prove the opposite inclusion we first observe that    the linear measure $ds = |dz|$ on $\T$ is a Carleson measure  for  $E^2(D_\delta)$. In other words
 \[
 \int_\T |f(z)|^2 ds \lesssim \|f\|^2_{E^2(D_\delta)}, \quad f\in E^2(D_\delta).
 \]
  Indeed, let $f \in E^2( D_\delta)$. It is sufficient to bound $\int_\T fh \, ds$ for $h \in L^2(\T)$ that are compactly supported in $\T \setminus \Sigma(\vartheta)$. For such $h$ we have, by Cauchy's formula 
  \begin{equation*}
\int_\T f(z)h(z) \, ds = \int_{\partial D_\delta} f(\zeta) \int_\T \frac{h(z)}{\zeta-z} \, ds(z) \, d\zeta.
\end{equation*}
The function  $H(\zeta) =  \int_\T \frac{h(z)}{\zeta-z} \, ds(z)$ belongs to $H^2_-(\D)$.  Since the arc length measure on $\partial \Dd$ is a Carleson measure for $H^2_-(\D)$ we have 
\[
\|H\|_{L^2(\partial D_\delta)} \lesssim \|H\|_{H^2_-(\D)} \lesssim \|h\|_{L^2(\T)}
\]  and therefore
\begin{equation*}
\left| \int_\T f(z)h(z) \, ds \right| \lesssim  \|f\|_{E^2(\partial D_\delta)} \|h\|_{L^2(\T)},
\end{equation*}
which is the required estimate.
   
The proposition now follows readily. Indeed, let $f\in \K_{{\vt}}=E^2(D_\delta)\cap {\vt}E^2_-(D_\delta)$, so that $f|_\D \in H^2(\D)$ and
$f=\vt g$, where $g\in E_-^2(D_\delta)$. Since $f$ and $\vt$ are holomorphic and $1\leq |\vt(z)|\leq \delta$ in $\Dd\setminus\D$, we see that $g$ admits prolongation to a function in $H_-^2(\D)$. Hence $f|_\D\in H^2(\D) \cap \vt H_-^2(\D)= K_\vt$.
 \end{proof} 
 %%%%%%%%%%%%%%%%
%%%%%%%%%%%%%%%%%%%

\section{Operator extension}
\subsection{Embeddings of $ E^2(D_\delta)$} We have seen that  $K_\vt$ can be considered as a subspace of $E^2(\Dd)$.
The following theorem reduces trace ideal problems for embeddings of $K_\vt$ to trace ideal problems for embeddings of the whole space $E^2(D_\delta)$. For a positive measure $\mu$, we denote by $I_\mu$ the embedding operator (inclusion map) from a space into $L^2(\mu)$.

\begin{theorem}\label{th:emb}
Let $\vt$ be a one-component inner function, $\mu$ be a positive measure on $\overline{\DD} \setminus \Sigma(\vt)$, and $0 < p < \infty$. Then the embedding $I_\mu: K_\vt\rightarrow L^2(\mu)$ is bounded, compact, or belongs to $\mS_p$ if and only if $I_\mu : E^2(D_\delta)\rightarrow L^2(\mu)$ is bounded, compact, or belongs to $\mS_p$, respectively.
\end{theorem}

\begin{proof} 
We focus on showing that $I_\mu: K_\vt\rightarrow L^2(\mu)$ belongs to $\mS_p$ if and only if $I_\mu : E^2(D_\delta)\rightarrow L^2(\mu)$ does. The statements concerning boundedness and compactness have very similar, but slightly easier proofs. We may further assume that $\dist (\Sigma(\vt), \supp \mu) > 0$, as one can easily see that all estimates are uniform with respect to this quantity.  

By Proposition \ref{l:kth},  $I_\mu:K_\vt\rightarrow L^2(\mu)$ belongs to $\mS_p$ if and only if $I_\mu:\K_{{\vt}} \rightarrow L^2(\mu)$ does. From Corollary \ref{cor:proj} follows the decomposition
\[
E^2(D_\delta)=\K_{{\vt}}\dotplus {\vt}E^2(D_\delta),
\]
with bounded projectors $P_1=M_{ {\vt}}P_-M_{ {\vt}^{-1}}$ and $Q_1= M_{ {\vt}}P_+M_{{\vt}^{-1}}$
 onto $\K_{{\vt}}$ and 
${\vt}E^2(D_\delta)$, respectively. Here $P_\pm$ are  defined in  \eqref{proj}, and    $M_h$ is the multiplication operator with symbol $h$.
We identify the functions in $E^2(D_\delta)$ with their boundary values and consider by extension $P_1$ and $Q_1$ as operators on $L^2(\partial \Dd)$. 
Let $C=\max(\|P_+\|, \|P_-\|)$. Since $ |{\vt}|=\delta$ a.e. on $\partial D_\delta$, we have $\|P_1\|, \|Q_1\|\le C$ and also 
$\|P_+M_{\vt^{-1}} \|\leq a:=C\delta^{-1}$. We estimate the singular value of the second summand in the decomposition
\begin{equation}
\label{dec}
I_\mu=I_\mu P_1+  I_\mu Q_1
\end{equation}

Introducing $\tilde{\vt}(z)={\vt}(z) \mathbf 1|_{\supp \mu}$, we have $I_\mu Q_1 = M_{\tilde {\vt}} I_\mu P_+M_{\vt^{-1}}$.  We obtain the following estimate of the $j$'th singular value of $I_\mu Q_1$:
\[
s_j(I_\mu Q_1)\leq \|M_{\tilde{\vt}} \| s_j(I_\mu) \| P_+M_{\vt^{-1}} \| \leq a s_j(I_\mu).
\]
Here we used the fact that $|\tilde{\vt}(z)|\leq 1$, $z\in \supp \mu$.  

Combining \eqref{dec}  with known (see e.g. \cite{GK}) inequalities for singular numbers we obtain 
 \[
\sum_{1}^n s_j(I_\mu) 
             \le \sum_{1}^n s_j(I_\mu P_1)+\sum_{1}^n s_j(I_\mu Q_1)
\le  \sum_{1}^n s_j(I_\mu P_1)+a  \sum_{1}^n s_j(I_\mu) .
\]

If $a<1$  and $p\ge 1$ this yields that $$\|I_\mu:E^2(\Dd) \to L^2(\mu)\|_{\mS_p} \sim \|I_\mu:  \K_\vt  \to L^2(\mu)\|_{\mS_p}.$$ More generally, such an equivalence of ideal norms holds for any {\em symmetrically normed ideal} of compact operators \cite{GK}. For $p<1$ note that $s_{2j-1}(I_\mu)\le s_j(I_\mu P_1)+as_j(I_\mu)$. If $a < a_0$ is small enough, where $2^{p+1}a_0^{p}=1$, we see that  
 $$\sum_j s_j(I_\mu P_1)^{p} \sim \sum_j s_j(I_\mu)^{p},$$
 finishing the proof also in this case.

To deal with general values of $a$,  we note that $I_\mu:K_\vt\rightarrow L^2(\mu)$ is in $\mS_p$ if and only if $I_\mu:K_{\vt^2}\rightarrow L^2(\mu)$ is in $\mS_p$, since $K_\vt\subset K_{\vt^2}=K_\vt\oplus\vt K_\vt$. Replacing $\vt$ by a sufficiently large power $\vt^n$ we will obtain a new value $a_1 = C\delta^{-n}$ such that
 $\|P_+ M_{\vt^{-n}}\| \leq a_1 < a_0 <1$; note that in moving from the study of $K_\vt$ to that of $K_{\vt^n}$ we do not change the domain $D_\delta$, so that the projections $P_\pm$ stay the same for all values of $n$. 
\end{proof}

%\subsection{More general ideals}To be written

\subsection{Whitney decomposition} 

To pass from the domain $\Dd$ to the unit disk, let $\sigma$ be a conformal mapping of $\D$ onto $\Dd$, and $\psi $ be its inverse. 
For $f\in E^2(D_\delta)$, let $h_f(w)=f(\sigma(w))(\sigma'(w))^{1/2}$, $w\in\DD$. 
Let further $\nu$ denote the measure on $\D$ given by
\[
\nu(E)=\int_{\sigma(E)}|\psi'|d\mu,\quad E\subset \DD.
  \]

Then for $f, g \in E^2(D_\delta)$ we have
\[
(I_\mu^*I_\mu f,g)=\int_{\overline{\D}\setminus \Sigma(\vt)} f(z)\overline{g}(z)d\mu(z)=\int_\DD h_f(w)\overline{h_g}(w)d\nu(w).
\]
%In light of this observation and Theorem \ref{th:emb} we find that
 Therefore, by Theorem \ref{th:emb}, the embedding $I_\mu:K_\vt\rightarrow L^2(\mu)$ is in $\mS_p$ 
 if and only if the embedding $I_\nu:H^2(\DD)\rightarrow L^2(\nu)$ is in $\mS_p$.
We can now apply the results from Luecking \cite{L}  to
  conclude that $I_\mu:K_\vt\rightarrow L^2(\mu)$ is in the Schatten ideal $\mS_p$, $0 < p < \infty$, if and only if
\begin{equation}\label{eq:Schnu}
\sum_j\left(\frac{\nu(R_j)}{d(R_j)}\right)^{p/2}<\infty,
\end{equation}
where $\{R_j\}$ is the standard dyadic decomposition of the unit disk, and $d(R_j) = \diam(R_j)$.

For a domain $\Omega$ in the plane, we say that a family of Borel sets $\{G_i\}_i$ is a {\em Whitney-type decomposition} of $\Omega$ if $\Omega=\cup_i G_i$, the covering is of finite multiplicity, and there exist constants $a,b,c$ such that: (i) if $z_1, z_2\in G_i$ then $\dist(z_1,\partial\Omega)\le c\dist(z_2,\partial\Omega)$, (ii) for each $i$ there exists $z\in G_i$ such that $B(z, ad)\subset G_i\subset B(z,bd)$, where $d=\dist(z,\partial\Omega)$.

We need  the following simple observation:

{\em Given two Whitney-type decompositions $\{G_i\}_i$ and $\{F_j\}_j$   of a domain $\Omega$, let 
$J(i)=\{j, G_i\cap F_j\neq \emptyset\}$. Then $M:=\sup_i |J(i)|< \infty.$}

 Indeed, just note that all $F_j$ that intersect $G_i$ have diameter proportional to $d_i=\dist(G_i,\partial\Omega)$ and the area of each such $F_j$ is proportional to $d_i^2$.

  Furthermore, if $\phi : \Omega \to \mathbb{C}$ is univalent, then $\{\phi(G_i)\}_i$ is a Whitney-type decomposition of $\phi(\Omega)$. This follows from standard estimates for univalent functions, see e.g. \cite{Pbook}.  Together with Luecking's condition \eqref{eq:Schnu} this yields the following corollary.
     \begin{corollary} \label{cor:discretecond}
Let $\vt$ be a one-component inner function, $\Dd$ be a corresponding level set and $\mu$ be a positive measure on  $\overline{\DD} \setminus \Sigma(\vt)$. Further,
let $\{G_i\}$ be any Whitney-type decomposition of $D_\delta$.
Then the embedding $K_\vt\rightarrow L^2(\mu)$ belongs to the Schatten ideal $\mS_p$, $0 < p < \infty$, if and only if
\begin{equation}\label{eq:mu}
\sum_i\left(\frac{\mu(G_i)}{d(G_i)}\right)^{p/2}<\infty,
\end{equation}
where $d(G_i)=\diam(G_i)$.
\end{corollary}

\begin{proof}
First let $Q_j=\sigma(R_j)$, where $\sigma:\DD\rightarrow D_\delta$ is a  conformal mapping, as above. Then $\{Q_j\}_j$ is a Whitney-type decomposition of $D_\delta$  and since $|\sigma'|\sim \dist(Q_j,\partial D_\delta)/\dist(R_j,\partial \DD)\sim \diam(Q_j)/\diam R_j$ (see \cite{Pbook}), (\ref{eq:Schnu}) is equivalent to
\[
\sum_i\left(\frac{\mu(Q_j)}{d(Q_j)}\right)^{p/2}<\infty.\]
But for any $\alpha>0$ we have 
\[
\sum_i\left(\frac{\mu(G_i)}{d(G_i)}\right)^{\alpha}\sim \sum_{i,j}\frac{\mu(G_i\cap Q_j)^{\alpha}}{d(G_i)^{\alpha}}\sim
\sum_{i,j}\frac{\mu(G_i\cap Q_j)^{\alpha}}{d(Q_j)^{\alpha}}\sim\sum_i\left(\frac{\mu(Q_j)}{d(Q_j)}\right)^{\alpha},\]
proving the corollary.
\end{proof}

We remark that Corollary \ref{cor:discretecond}, which will be the main ingredient in the proof of Theorem \ref{thm:integralcond} below, can be deduced from a result of Baranov \cite{B}. We think however that Theorem \ref{th:emb} may be of independent interest and it has some applications which do not appear to us to be immediate consequences of Corollary \ref{cor:discretecond}, see Section \ref{subsec:weighted}.

The proof of Corollary \ref{cor:discretecond} also gives a simple and natural criterion for the Schatten class memberships of embeddings of the Hardy-Smirnov space $E^2(D)$ into $L^2(\mu)$ for measures $\mu$ in $D$; one has to check the Luecking condition for an arbitrary Whitney-type decomposition of $D$.

%%%%%%%%%%%%%%%%%%%%%%%%

%%%%%%%%%%%%%%%%%%%%%%%%%%%%%%%%%%
\section{Composition operators}
 
\subsection{Preliminaries}
For a holomorphic function $\vf: \DD \to \DD$, we denote by
 $
 C_\vf: f \mapsto f\circ \vf
$ the composition operator acting on holomorphic functions $f$ in $\DD$.  This operator is bounded on the Hardy space 
 $H^2(\DD)$ (see e.g. \cite{S}). 
We study the operator    $C_\vf : K_\vt \to H^2(\DD)$, where $\vt$ is an inner function in $\DD$.
The compactness of $C_\vf$  in terms of the Nevanlinna counting function
\[
N_\vf(z)=\sum_{\vf(\zeta)=z} \log \frac 1 {|\zeta|}
\]
was characterized in \cite{LM}; $C_\vf:K_\vt\rightarrow H^2$ is compact if and only if 
\begin{equation}\label{eq:comp}
\limsup_{|z|\rightarrow 1}\frac{N_\vf(z)(1-|\vt(z)|^2)}{1-|z|^2}=0.
\end{equation}
 The basic tools in the argument are the Stanton formula
 \beq
 \label{eq:nvf}
 \|C_\vf f\|^2= 2 \int_\DD |f'(z)|^2 N_\vf(z) dA(z) + |f(\vf(0))|^2,
 \eeq
where $A$ is the normalized area measure, and  also the norm inequality due to Axler, Chang and Sarason \cite{ACS}
\begin{equation} \label{eq:axler}
\int_{\DD} |f'(z)|^2\frac{1-|z|^2}{(1-|\vt(z)|^2)^b}\le C\|f\|_2^2,\quad f\in K_\vt,\quad b\in(0,1/2).
\end{equation}

In this section we discuss when $C_\vf$ belongs to the Schatten ideals $\mS_p$ in the one-component case,  aiming  to capture the interaction between the symbol $\vf$ and the inner function $\vt$ that defines the model space. We recall the known description of the Schatten ideals for composition operators on the whole of $H^2$, due to Luecking and Zhu \cite{LZ}. The operator $C_\vf$ belongs to $\mS_p(H^2)$  if and only if
\begin{equation}\label{eq:Sp}
\int_{\DD}\left(\frac{N_\vf(z)}{1-|z|^2}\right)^{p/2}\frac{dA(z)}{(1-|z|^2)^2}<\infty.
\end{equation}   
%The proof given in \cite{LZ} implies the following slightly more general result. Let $w$ be a subharmonic function on $\DD\setminus\{0\}$, then the differentiation operator $D:H^2(\DD)\rightarrow L^2(wdA)$ belongs to the Schatten ideal $S_p$ if and only if
%\begin{equation}\label{eq:Spw}
%\int_{\DD}\left(\frac{w(z)}{1-|z|^2}\right)^{p/2}\frac{dA(z)}{(1-|z|^2)^2}<\infty.
%\end{equation}

It is well understood that the composition operators can be considered as a special case of the embedding operators, see e.g. \cite{CMc}.   We shall now clarify this connection in our context, so that we may apply Theorem \ref{th:emb} and Corollary \ref{cor:discretecond}. We associate $\vf : \D \to \D$ with its pullback measure $\mu_\vf$ on $\overline{\D}$; $$\mu_\vf(E) = s(\vf^{-1}(E) \cap \mathbb{T}), \quad E \subset \overline{\D},$$ where $s$ denotes the Lebesgue measure on $\T$. 

It is clear that $C_\vf : K_\vt \to H^2$ is unitarily equivalent to the embedding operator $I_{\mu_\vf} : K_\vt \to L^2(\mu_\vf)$, and similarly that $C_\vf : E^2(D_\delta) \to H^2$ is equivalent to $I_{\mu_\vf} : E^2(D_\delta) \to L^2(\mu_\vf)$. Before applying Theorem \ref{th:emb} we need to verify that $\mu_\vf(\Sigma(\vt)) = 0$. This is true in view of the following lemma and the fact that $\Sigma(\vt)$ has zero linear Lebesgue measure when $\vt$ is one-component, see \cite{A}. 
\begin{lemma}
$\mu_\vf|_{\T}$ is absolutely continuous.
\end{lemma}
\begin{proof}
It is sufficient to verify that $\mu_\vf(E) = 0$ for every closed measure zero set $E \subset \T$. We follow the approach of the original proof of the F. and M. Riesz theorem. Namely, there exists a continuous function $G : \overline{\D} \to \overline{\D}$, holomorphic in $\D$, such that $G(z) = 1$ for $z \in E$ and $|G(z)| < 1$ for $z \in \overline{\D} \setminus E$. Then $\lim_{k} \int_\D G^k d\mu_\vf = \mu_\vf(E)$. On the other hand, the sequence $G^k \circ \vf$ converges pointwise to zero in $\D$ and  is uniformly bounded.  Therefore,
\begin{equation*}
\mu_\vf(E) = \lim_{k} \int_\D G^k \, d\mu_\vf = \lim_k \int_\T G^k \circ \vf \,ds =\lim_k \left(G(\vf(0))\right)^k= 0.
\end{equation*}
\end{proof}

Theorem \ref{th:emb} now yields 
\begin{corollary} \label{cor:compdiscrete}
Let $\vt$ be a one-component inner function and $\vf : \D \to \D$ be a holomorphic function. Given any  Whitney-type decomposition $\{G_j\}$ of $D_\delta$,  the operator $C_\vf : K_\vt \to H^2$ belongs to $\mS_p$, $0 < p < \infty$, if and only if 
\begin{equation} \label{eq:discretecond2}
\sum_i \left( \frac{\mu_\vf(G_i)}{d(G_i)} \right)^{p/2} < \infty.
\end{equation}
\end{corollary}

\subsection{Nevanlinna counting function}
In this subsection we implement the approach of \cite{LQLR, LQLR2} in our more general setting, with the goal of showing that \eqref{eq:discretecond2}  is equivalent to 
 \[
\int_{D_\delta}\left(\frac{N_\vf(z)}{\dist(z, \partial D_\delta)}\right)^{p/2}\frac{dA(z)}{\dist(z,\partial D_\delta)^2}<\infty.
\]
Theorem \ref{thm:integralcond} then immediately follows from the relation
\begin{equation}\label{eq:aleksandrov}
\dist(z,\partial D_\delta) \sim \frac{1-|z|^2}{1-|\vt(z)|^2}, \quad z \in \overline{\D} \setminus \Sigma(\vt),
\end{equation} 
 which holds for one-component inner functions, by Theorems 1.1 and 1.2 in Aleksandrov \cite{A}.

We begin by fixing a convenient Whitney-type decomposition of $D_\delta$. Clearly we are interested only in domains that intersect $\overline{\D}$.
We construct a decomposition of $ A = \{1/2<|z|\leq 1\} \setminus \Sigma(\vt)$ as follows. First we divide $A$ into four equal parts, one for each quadrant. Each part is roughly a Carleson square. Fix some $\gamma>0$. We  say that a Carleson square $S$ is {\em good} if $\dist(S, \partial D_\delta)>\gamma d(S)$.  If a square is good we include  it in our family of sets $\{G_i\}$. Otherwise we include  its upper half  into the family $\{G_i\}$
and divide the lower half into two new Carleson squares. We repeat the procedure inductively, obtaining a countable family of sets $\{G_i\}$ that covers $A$. In particular, every point $z \in \T \setminus \Sigma(\vt)$ is included in a good square, since $\Sigma(\vt)$ is a closed set. We claim that $\dist(G_i, \partial D_\delta) \simeq d(G_i)$ for each $G_i$. For if $G_i$ is the upper half of a bad square $S$, it automatically satisfies $d(G_i) \lesssim \dist(G_i, \partial D_\delta)$. Since $S$ was bad we obtain the reverse inequality, 
    $$
    \dist(G_i, \partial D_\delta) \lesssim d(G_i) + \dist(S, \partial D_\delta) \leq  d(G_i) + \gamma d(S) \lesssim d(G_i).
    $$ 
A similar argument works for the good squares $G_i$.  Further, $\{G_i\}$ can be extended to a Whitney-type decomposition of the whole $D_\delta$,  since $\supp\mu\subset \overline{\D}$.  We omit the corresponding  terms.

For each $G_i$ we let $W_i$  be the corresponding Carleson square (i.e. either $W_i=G_i$ or $G_i$ is the upper half of $W_i$). Given  a Carleson square $W$ supported by the arc $I \subset \T$  and $a>0$ we denote by $aW$ the Carleson square supported by the arc $aI\subset \T$. $aI$ has the same center as $I$ and $|aI|=a|I|$.     
According to \cite{LQLR} and \cite{LQLR2}  there exists $a>1$ such that for some   constant $c > 0$
the  following estimates hold,
\begin{equation}
 \label{eq:munev1}
\mu(W_i)^{p/2}\le \frac{c}{A(W_i)}\int_{\tilde{W}_i}N_\vf^{p/2}dA,
\end{equation}
\begin{equation} 
\label{eq:munev2}
\sup_{W_i}N_\vf\le c\mu(\tilde{W}_i),
\end{equation}
 where $\tilde{W_i}=a {W_i}$.

To obtain the integral condition, we follow the argument of \cite[Proposition 3.3]{LQLR0} and prove that
 (\ref{eq:discretecond2}) for the above decomposition $\{G_i\}_i$ is equivalent to
\begin{equation}\label{eq:mu1}
\sum_i\left(\frac{\mu_\vf(\tilde{W}_i)}{d(W_i)}\right)^{p/2}<\infty.
\end{equation}

Clearly (\ref{eq:mu1}) implies (\ref{eq:discretecond2}). Now we prove the converse statement.

For each $j$  and $n$ let  $I(j)=\{i \, : \, \tilde{W}_i\cap G_j\neq\emptyset \}$ and 
\[
S_{n,j} = \{i \in I(j) \, : \, 2^nd(G_j) \leq d(W_i) < 2^{n+1}d(G_j) \}.
\]
 We note that $d(G_j) \sim \dist(G_j,\partial D_\delta)\lesssim d(W_i)$ for any $i\in I(j)$, yielding that 
 $S_{n,j} = \emptyset$ when $n<n_0$ for some fixed (negative) $n_0$.  We also have 
 $\dist(G_j, W_i)\lesssim d(W_i)$ which yields  $C:=\sup_{j,n} |S_{n,j}| < \infty$. 
 
 For $\alpha:=p/2\le1$ we now obtain
\begin{multline*}
\sum_i\left(\frac{\mu_\vf(\tilde{W}_i)}{d(W_i)}\right)^{\alpha}\le \sum_i\sum_{j: i\in I(j)}\frac{\mu_\vf(G_j)^\alpha}{d(W_i)^\alpha}=\\
\sum_j\mu_\vf(G_j)^\alpha\sum_{i\in I(j)} d(W_i)^{-\alpha}\lesssim \sum_j\left(\frac{\mu_\vf(G_j)}{d(G_j)}\right)^\alpha,
\end{multline*}
as desired. 

For $\alpha>1$ we also follow the argument of \cite{LQLR0}: for each index $i$, consider the sets
\begin{equation*}
s_{n,i} = \{j \, : \, i \in I(j), \; 2^{-n-1}d(W_i) \leq d(G_j) < 2^{-n}d(W_i) \}.
\end{equation*}
We have  $|s_{n,i}| \leq C  2^n$  for some constant $C$. Let $\beta$ be the conjugate exponent of $\alpha$ and choose   $\gamma\in(1-1/\alpha, 1)$ such  that $1 - \gamma \beta < 0$. By the H{\"o}lder inequality
\begin{multline*}
\mu_\vf (\tilde{W}_i)\le \left(\sum_{j: i\in I(j)} d(G_j)^{\gamma\beta}\right)^{1/\beta}\left(\sum_{j: i\in I(j)}d(G_j)^{-\gamma\alpha}\mu_\vf(G_j)^\alpha\right)^{1/\alpha}\lesssim\\
 d(W_i)^\gamma\left(\sum_{j: i\in I(j)}d(G_j)^{-\gamma\alpha}\mu_\vf(G_j)^\alpha\right)^{1/\alpha},
 \end{multline*}
so that finally,
\[
\sum_i\left(\frac{\mu_\vf(\tilde{W}_i)}{d(W_i)}\right)^{\alpha}\le\sum_j \mu_\vf(G_j)^\alpha d(G_j)^{-\gamma\alpha}\sum_{i\in I(j)} d(W_i)^{\gamma\alpha-\alpha}\lesssim \sum_j\left(\frac{\mu_\vf(G_j)}{d(G_j)}\right)^\alpha.
\]

The sought after criterion in terms of the Nevanlinna counting function now follows readily.

\begin{theorem} \label{thm:integralcond}
Let $\vt$ be a one-component inner function. The operator $C_\vf:K_\vt\rightarrow H^2$ is in $\mS_p$, $0 < p < \infty$, if and only if 
\[
\int_\DD\left(\frac{N_\vf(z)(1-|\vt(z)|)^2}{1-|z|^2}\right)^{p/2}\left(\frac{1-|\vt(z)|^2}{1-|z|^2}\right)^2dA<\infty.\]
\end{theorem}

\begin{proof}
We follow the proof for the Hardy space given in \cite[Theorem 6.1]{LQLR2}. Inequality \eqref{eq:munev2} for the Nevanlinna counting function implies
\begin{multline*}
\int_{\DD}\left(\frac{N_\vf(z)}{\dist(z, \partial D_\delta)}\right)^{p/2}\frac{dA(z)}{\dist(z,\partial D_\delta)^2}\lesssim \\ \sum_i
d(G_i)^{-2-p/2}\int_{G_i}N_\vf(z)^{p/2}dA(z)\lesssim\sum_i d(G_i)^{-p/2}\mu_\vf(\tilde{W}_i)^{p/2}.\end{multline*}
The converse follows from  inequality \eqref{eq:munev1}. Specifically,
\begin{align*}
\sum_i d(G_i)^{-p/2}\mu_\vf(G_i)^{p/2} &\le \sum_i d(G_i)^{-p/2-2}\int_{\tilde{W}_i}N_\vf(z)^{p/2}dA(z) \\ &\lesssim
\sum_j\sum_{i\in I(j)}d(G_i)^{-p/2-2}\int_{G_j}N_\vf(z)^{p/2}dA(z)\\ &\lesssim \int_{\DD}\left(\frac{N_\vf(z)}{\dist(z, \partial D_\delta)}\right)^{p/2}\frac{dA(z)}{\dist(z,\partial D_\delta)^2},\end{align*}
where the last inequality follows as in the discussion preceding the statement of the theorem.
\end{proof}

%%%%%%%%%%%%%%%%%%%%%%%%%%%%%%%%%%%%%%%%%%%%%%%%

\subsection{Weighted composition operators on $H^2$} \label{subsec:weighted}

We complete the study of composition operators on model spaces generated by one-component inner functions, by establishing the connection given by Theorem \ref{th:emb} between composition operators on $K_\vt$ and weighted composition operators on $H^2(\D)$. As previously noted, $C_\vf : K_\vt \to H^2$ is of $\mathcal{S}_p$-class if and only if $C_\vf : E^2(\Dd) \to H^2$ has the same property. That is, if and only if the weighted composition operator $C : H^2 \to H^2$ is of $\mathcal{S}_p$-class, where
\[C h=(\psi'\circ \vf )^{1/2}h\circ\psi\circ\vf, \quad h \in H^2.\]
We mention \cite{M, GGN}, where composition operators on Hardy-Smirnov spaces $E^2(D)$ have been studied as weighted composition operators on $H^2(\DD)$. Note that $ C h = h(M_{\psi\circ\vf}) (\psi'\circ \vf )^{1/2}$, at least for polynomials $h$. Here $M_{\psi\circ\vf}$ denotes a multiplication operator on $H^2$, and we hence understand $Ch$ as the action of the operator $h(M_{\psi\circ\vf})$ on $(\psi'\circ \vf )^{1/2} \in H^2$. For $p \geq 1$, Harper and Smith \cite{HS} have utilized the theory of contractive semigroups to characterize the Schatten membership of such operators in terms of Berezin transform-type conditions. Note that in their notation, $C = \Lambda_{M_{\psi\circ\vf}, (\psi'\circ \vf )^{1/2}}$.
\begin{theorem} \label{thm:berezin}
Denote by $G_z(w) = \frac{(1-|z|^2)^{1/2}}{1-\bar{z}w}$ the normalized reproducing kernel of $H^2$ at $z$ and by $H_z(w) = \frac{(1-|z|^2)^{3/2}w}{(1-\bar{z}w)^2}$ the normalized derivative of the reproducing kernel at $z$. Let $\vt$ be a one-component inner function and let $C$ be defined as above.
\begin{enumerate}
\item If $1 \leq p \leq 2$, then $C_\vf : K_\vt \to H^2$ belongs to $\mathcal{S}_p$ if and only if 
\begin{equation*}
\int_\D \| C H_z \|^p_{H^2} \frac{dA(z)}{(1-|z|^2)^2} < \infty.
\end{equation*}

\item If $2 < p < \infty$, then $C_\vf : K_\vt \to H^2$ belongs to $\mathcal{S}_p$ if and only if 
\begin{equation*}
\int_\D \| C G_z \|^p_{H^2} \frac{dA(z)}{(1-|z|^2)^2} < \infty.
\end{equation*}
\end{enumerate}
\end{theorem}
\begin{remark}
For $\vt = 0$, $p \geq 2$, we obtain the characterization of the Schatten classes in terms of the Berezin transforms found in \cite{Z}. For $p = 2$, note that 
\begin{align*}
\int_\D \| C H_z \|^2_{H^2} \frac{dA(z)}{(1-|z|^2)^2} &= \int_{\overline{\D}} |\psi'(w)| |\psi(w)|^2 \int_\D \frac{1-|z|^2}{|1 - \bar{z}\psi(w)|^4} dA(z) \, d\mu_\vf (w) \\ &\sim \int_{\overline{\D}} \frac{|\psi'(w)| |\psi(w)|^2}{1-|\psi(w)|^2} d\mu_\vf (w).
\end{align*}
Note that $1 - |\psi(w)|^2 \sim \dist(w, \partial D_\delta) |\psi'(w)|$ by standard estimates \cite{Pbook}, and that we may reexpress \eqref{eq:aleksandrov} as
\begin{equation*}
\dist(w,\partial D_\delta) \sim \frac{1-|w|^2}{1-|\vt(w)|^2} = \|k_w\|^{-2}, \quad w \in \overline{\D} \setminus \Sigma(\vt),
\end{equation*}
 where $k_w$ is the reproducing kernel of $K_\vt$ at $w$. In summary, 
\begin{equation*}
\int_\D \| C H_z \|^2_{H^2} \frac{dA(z)}{(1-|z|^2)^2} \sim \int_{\overline{\D}} \|k_w\|^{2} \, d\mu_\vf(w).
\end{equation*}
Theorem \ref{thm:berezin} could hence be viewed as a generalization of the simple fact that $C_\vf : K_\vt \to H^2$ is Hilbert-Schmidt if and only if $ \int_{\overline{\D}} \|k_w\|^{2} \, d\mu_\vf < \infty$.
\end{remark}

\subsection{Examples on the Paley--Wiener space} 
In this section  we consider the special case $\vt(z) = \exp \left ({-\frac{1+z}{1-z}}\right ) $. The space  $K_\vt$ can then be 
naturally identified    with the classical Paley--Wiener space  of entire functions.

For $0 < \alpha < 1$ we construct domains $U_\alpha \subset \DD$ satisfying the following proposition. The domain $U_\alpha$ will be chosen so that the boundary $\partial U_\alpha$ intersects $\TT$ only at $1$ and has a corner of angle $\pi \alpha$ there. 
\begin{proposition} \label{prop:pwex}
 For each $\alpha$, $0 < \alpha < 1$, there exists a domain $U_\alpha$ such that any corresponding Riemann map $\vf_\alpha : \DD \to U_\alpha$ satisfies
\begin{enumerate}
\item $C_{\vf_\alpha} : H^2 \to H^2$ is compact but not in any Schatten class $\mS_p$,
\item $C_{\vf_\alpha} : K_\vt \to H^2$ is in $\mS_p$ for $p > 2\alpha/(\alpha-1)$.
\end{enumerate}
\end{proposition}
Note that $C_{\vf_\alpha} : K_\vt \to H^2$ is in $\mS_p$ for one Riemann map $\vf_\alpha : \DD \to U_\alpha$ if and only if it is in $\mS_p$ for all such Riemann maps. The same statement is obviously true for $C_{\vf_\alpha} : H^2 \to H^2$.

Let $V_\alpha \subset \DD$ be the simply connected domain whose boundary consists of the upper half-circle $\{z \, : \, |z| = 1, \, \Im z \geq 0\}$ and the circular arc with terminal points $-1$ and $1$, intersecting the upper half-circle at an interior angle $\pi \alpha$ at those points. Note that $V_\alpha = \kappa(A_\alpha)$, where 
\begin{equation*}
A_\alpha = \{r e^{i\theta} \, : \, 0 < r < \infty, \, 0 < \theta < \pi \alpha \},
\end{equation*}
and $\kappa$ is the M\"obius map
\begin{equation*}
\kappa(w) = \frac{1+iw}{1-iw}.
\end{equation*}

Let $\psi_\alpha : \DD \to V_\alpha$ be a Riemann map. The next lemma says that $C_{\psi_\alpha} : K_\vt \to H^2$ is in $\mS_p$ if and only if $p > 2\alpha/(\alpha -1)$. This will immediately imply item (2) of Proposition \ref{prop:pwex}, as we will choose $U_\alpha$ as a subset of $V_\alpha$, so that $C_{\vf_\alpha} = C_\tau C_{\psi_\alpha}$ for a composition operator $C_{\tau} : H^2 \to H^2$ with $\tau$ a Riemann map from $\D$ to $\psi_\alpha^{-1}(U_\alpha)$.
\begin{lemma}
$C_{\psi_\alpha} : K_\vt \to H^2$ is $p$-Schatten if and only if $p > 2\alpha/(\alpha -1)$.
\end{lemma}
\begin{proof}
Y. Zhu \cite{yzhu01} obtains the corresponding result on $H^2$, for domains similar to $V_\alpha$, rescaled and translated to only touch (tangentially from one side) the unit circle at $z = 1$. In view of the fact that $C_{\psi_\alpha} : K_\vt \to H^2$ belongs to $\mS_p$ if and only if $C_{\psi_\alpha} : E^2(\Dd) \to H^2$ belongs to $\mS_p$ and the observation that $\partial D_\delta$ in this case is a circle such that $\T$ is internally tangent to $\partial D_\delta$ at 1, the lemma follows. 
A direct computational proof based on our integral criterion can be also given.
\end{proof}

We now define $U_\alpha$. Let $\gamma_\alpha$ be the circular arc of $\partial V_\alpha$ which is not the upper half-circle, $\gamma_\alpha = \{z \in \partial V_\alpha \, : \, |z| \neq 1\} \cup \{-1,1\}$. Close to $1$, we choose $\partial U_\alpha$ to coincide with the union of $\gamma_\alpha$ and
\begin{equation*}
\tau = \{re^{i\theta} \, : \, 0 < \theta \leq \pi/2, \, r = 1-e^{-1/\theta}\}.
\end{equation*}
That is, let $U_\alpha$ be a simply connected domain contained in $V_\alpha$, such that $\partial U_\alpha \cap \T = \{1\}$ and for sufficiently small $\epsilon > 0$ we have
\begin{equation*}
\partial U_\alpha \cap B_\epsilon(1) = (\gamma_\alpha \cup \tau) \cap B_\epsilon(1),
\end{equation*}
where $B_\epsilon(1)$ is a ball centered at $1$ of radius $\epsilon$.

\begin{proof}[Proof of Proposition \ref{prop:pwex}]
We have already proven the validity of item (2). It remains to prove that $C_{\vf_\alpha} : H^2 \to H^2$ is compact but fails to lie in any Schatten class $\mS_p$. For simplicity in presentation we will assume that $\alpha < 1/2$. The proof for $\alpha \geq 1/2$ is very similar and follows exactly the same ideas. Without loss of generality we further assume that $\vf_\alpha(1) = 1$. 

Let $\eta_\alpha = \vf_\alpha^{-1}$. For $z \in U_\alpha$ close to $1$ we have that $N_{\vf_\alpha}(z) \sim 1 - |\eta_\alpha(z)|$. We will again make use of the inequality
\begin{equation*}
1 - |\eta_\alpha(z)| \sim \dist(z, \partial U_\alpha) |\eta_\alpha'(z)|.
\end{equation*}
Since $\eta_\alpha$ maps $U_\alpha$, a domain with a corner of angle $\pi \alpha$ at $1$, conformally onto the unit disc, it follows that $|\eta_\alpha'(z)| \sim |z-1|^{1/\alpha -1}$ for $z$ close to $1$ (see e.g. \cite{Pbook}). Therefore, for $z \in U_\alpha$ close to $1$ we have
\begin{equation*}
N_{\vf_\alpha}(z) \sim \dist(z, \partial U_\alpha) |z-1|^{1/\alpha -1}.
\end{equation*}
This estimate obviously implies that $\lim_{|z| \to 1}\frac{N_{\vf_\alpha}(z)}{1-|z|} = 0$ so that $C_{\vf_\alpha}$ is compact on $H^2$ \cite{S}, and we will now use it to show that
\begin{equation*}
\int_{U_\alpha \cap B_\epsilon(1)} \frac{N_{\psi_\alpha}(z)^{p/2}}{(1-|z|^2)^{p/2+2}} \, dA(z) = \infty
\end{equation*}
for every $p > 0$, proving in view of \eqref{eq:Sp} the desired statement.

We will restrict ourselves to considering points in the set
\begin{equation*}
U_\alpha' = \{z = re^{i\theta} \in B_\epsilon(1) \, : \, r < 1 - 2e^{-1/\theta}, \, \theta > k(1-r)\},
\end{equation*}
 for a fixed large constant $k$. $k$ should be chosen so large and then $\epsilon$ so small that there exists a constant $C > 1/\tan(\pi\alpha)$ such that every point $re^{i\theta} \in U_\alpha'$ satisfies $r\sin \theta > C(1-r\cos\theta)$. In other words, $U_\alpha'$ should lie above the line $\Im z = C(1-\Re z)$, which intersects the unit circle in the point $1$ at an angle smaller than $\pi \alpha$. If necessary, we decrease $\epsilon$ further, attaining that $U_\alpha' \subset U_\alpha$ and 
\begin{equation*}
\dist(z, \partial U_\alpha) \sim 1 - e^{-1/\theta} - r \sim 1 - |z|, \quad z \in U_\alpha'.
\end{equation*}

The proof is now finished by the following chain of inequalities.
\begin{multline*}
\int_{U_\alpha'} \frac{N_{\psi_\alpha}(z)^{p/2}}{(1-|z|^2)^{p/2+2}} \, dA(z) \sim \int_{U_\alpha'} \frac{|1-z|^{\frac{p(1-\alpha)}{2\alpha}}}{(1-|z|^2)^{2}} \, dA(z) \\ \gtrsim \int_{1-\epsilon'}^1 \int_{k(1-r)}^{\frac{1}{\log \frac{2}{1-r}}} \frac{(1-2r\cos\theta + r^2)^{\frac{p(1-\alpha)}{4\alpha}}}{(1-r)^2} \, d\theta \, dr \\ \gtrsim  \int_{1-\epsilon'}^1 \int_{k(1-r)}^{\frac{1}{\log \frac{2}{1-r}}} \frac{\theta^{\frac{p(1-\alpha)}{2\alpha}}}{(1-r)^2} \, d\theta \, dr 
\gtrsim \int_{1-\epsilon'}^1 \frac{\left(\frac{1}{\log \frac{2}{1-r}}\right)^{\frac{p(1-\alpha)}{2\alpha}+1}}{(1-r)^2} \, d\theta \, dr
= \infty,
\end{multline*}
where $0 < \epsilon' < \epsilon$ is sufficiently small.
\end{proof}

 %%%%%%%%%%%%%%%%%%%%%%%%%%
\section{ Remarks on composition operators on general model spaces}
In this section we do not assume that the inner function $\vt$ is 
one-component.
\subsection{The Hilbert-Schmidt norm}   
Denote the reproducing kernel of $K_\vt$ at $w$ by $k(w, \, \cdot \,)$,
\[
k(w,z)=\frac{1-\overline{\vt(w)}\vt(z)}{1-\bar{w}z}, \quad z \in \D.
\]
\begin{proposition}\label{p:HS}
The Hilbert-Schmidt norm of the composition operator $C_\vf: K_\vt\rightarrow H^2$ is given by
the expression
\[
\|C_\vf\|^2_{HS}=1-|\vt(0)|^2+\frac12\int_{\DD}\Delta k(z,z) N_\vf(z)dA(z).
\] 
\end{proposition} 

\begin{proof}
Let $\{e_j(z)\}$ be any orthonormal basis in $K_\vt$. Then $k(z,z)=\sum_j|e_j(z)|^2$ and the Stanton formula \eqref{eq:nvf} yields
\[
\|C_\vf\|^2_{HS}=\sum_j(C_\vf e_j, C_\vf e_j)=\sum_j \left(|e_j(0)|^2+2\int_\DD |e_j'|^2N_\vf dA\right).\]
Since $\sum_j|e_j'(z)|^2= \frac{1}{4}\Delta\sum_j|e_j(z)|^2$, we obtain
\[
\|C_\vf\|_{HS}^2 =k(0,0)+\frac12\int_\DD \Delta k(z,z) N_\vf(z)dA(z).\]
\end{proof}

\begin{corollary} 
(i)  If 
$$
\int_\DD\frac{1-|\vt(z)|^2}{(1-|z|^2)^3}N_\vf(z) \, dA(z)<\infty
$$ then $C_\vf:K_\vt\rightarrow H^2$ is a Hilbert-Schmidt operator;

(ii) If $C_\vf:K_\vt\rightarrow H^2$ is a Hilbert-Schmidt operator then 
$$
\int_\DD\frac{(1-|\vt(z)|^2)^3}{(1-|z|^2)^3}N_\vf(z) \, dA(z)<\infty.
$$
\end{corollary}
\begin{proof} 
 A direct calculation shows that
\[
\frac14\Delta k(z,z)=\frac{(1+|z|^2)(1-|\vt(z)|^2)}{(1-|z|^2)^3}-2\frac{\Re(z\overline{\vt(z)}\vt'(z))}{(1-|z|^2)^2}-\frac{|\vt'(z)|^2}{1-|z|^2}.
\]
Applying the standard inequality $|\vt'(z)|\le (1-|\vt(z)|^2)(1-|z|^2)^{-1}$, 
\begin{equation}
\label{eq:DK}
4\frac{(|z|-|\vt(z)|)^2(1-|\vt(z)|^2)}{(1-|z|^2)^3}\le\Delta k(z,z)\le 4\frac{(1+|z|)^2(1-|\vt(z)|^2)}{(1-|z|^2)^3}.
\end{equation}
 
 Statement {\em (i)}  is an immediate consequence of the right hand side inequality in \eqref{eq:DK}.

In order to prove statement {\em(ii)}  we first observe that  if  $C_\vf:K_{\vt}\rightarrow H^2$ is a Hilbert-Schmidt operator then so is $C_\vf:K_{z\vt}\rightarrow H^2$,  since $K_{z\vt}=K_z\oplus z K_\vt$ and $\dim K_z=1$.  Denote by 
$\tilde{k}(w,z)$ the reproducing kernel of $K_{z\vt}$ at $w$. We then have
\begin{multline*}
\int_{\DD\setminus\frac{1}{2}\DD}\frac{(1-|\vt(z)|^2)^3}{(1-|z|^2)^3}N_\vf(z) \le \\ 4\int_{\DD}\frac{(1-|z\vt(z)|^2)(|z|-|z\vt(z)|)^2}{(1-|z|^2)^3}N_\vf(z) dA(z) \le\\
\int_\DD\Delta \tilde{k}(z,z) N_\vf(z)dA(z)\le 2 \|C_\vf:K_{z\vt}\rightarrow H^2\|_{HS}^2.
\end{multline*}
\end{proof}

\begin{remark} 
In the previous section we proved  that the inequality 
\begin{equation} \label{eq:hsonecomp}
\int_\DD\frac{(1-|\vt(z)|^2)^3}{(1-|z|^2)^3}N_\vf(z) \, dA(z) <\infty
\end{equation}
gives a complete  description of the Hilbert-Schmidt composition operators $C_\vf : K_\vt \to H^2$ in the case of one-component $\vt$. This condition is not sufficient in general,  see Section \ref{subsec:counterex}
below.
\end{remark}

  %%%%%%%%%%%%%%%%%%

\subsection{A sufficient condition for larger Schatten ideals}
The Hilbert-Schmidt norm characterization, discussed for composition operators so far, remains true for the differentiation operator acting between $K_\vt$ and $L^2(\DD, \mu)$, where $\mu$ is a   positive, finite measure  on $\D$. Using complex interpolation we can obtain  a sufficient  condition  for a composition operator to be in the Schatten ideal $\mS_p$ when $2\le p<\infty$. 

First we assume that $\Phi(z)$ is a positive function on $\DD$ such that the differentiation $D:K_\vt\rightarrow L^2(\Phi(z) \, dA)$ is a bounded operator.
It follows from \eqref{eq:axler} that we can take
\[
\Phi(z)=\frac{1-|z|^2}{(1-|\vt(z)|^2)^b},\]
with $b\in(0,1/2)$. For the special case when $\vt$ is a one-component inner function we may choose
\[
\Phi(z)=\frac{1-|z|^2}{1-|\vt(z)|^2},\]
see \cite[Theorem 1]{C2}.
\begin{proposition}\label{p:sufSp}
Let $\Phi(z)$ be as above and $p\ge 2$. If 
\[
\int_\DD \left(\frac{N_\vf(z)}{\Phi(z)}\right)^{p/2}\Delta k(z,z)\Phi(z) \, dA(z) < \infty
\]
then $C_\vf:K_\vt\rightarrow H^2$ belongs to $\mS_p$.
\end{proposition}
\begin{proof}
The proof follows an  idea of Luecking \cite{L}; see also \cite{B}.
 By the Stanton formula \eqref{eq:nvf}, we want to prove that the differentiation operator $D:K_\vt\rightarrow L^2(N_\vf \, dA)$ belongs to the Schatten ideal $\mS_p$.
  It is enough to prove the corresponding statement where $N_\vf$ has been replaced by $N=\chi_r N_\vf$,
   $\chi_r$ being the characteristic function of the disk $r\DD$, as long as the corresponding Schatten norm
    is  independent  of $r$. 

Consider the following holomorphic family of compact operators 
\[
(T_\zeta f)(z)=N(z)^{p(1-\zeta)/4}\Phi(z)^{1/2-p(1-\zeta)/4}f'(z),\]
where $\zeta\in\CC$,  $0\le \Re \zeta\le 1$. We note that $T_{1+iy}:K_\vt\rightarrow L^2(dA)$ is bounded 
by a constant $C_1$ that depends on $\Phi$ only. 
Further, $T_{iy}:K_\vt\rightarrow L^2(dA)$ is a Hilbert-Schmidt operator with norm (see the proof of Proposition \ref{p:HS})
\[
\|T_{iy}\|_{HS}\le 1+\int_{\DD}\Delta K(z,z) N(z)^{p/2}\Phi(z)^{1-p/2}dA(z)\le C_2
\]
by the assumption of the theorem. Then $T_{1-2/p}\in \mS_p$ and
 $\|T_{1-2/p}\|_{\mS_p}\le C_1^{2/p}C_2^{1-2/p}$ by Theorem 13.1 in \cite{GK}. But $T_{1-2/p}f(z)=N(z)^{1/2}f'(z)$.
\end{proof}

Note that for $\vt=0$,  Proposition \ref{p:sufSp} gives us the known condition of Luecking and Zhu \cite{LZ}.

%%%%%%%%%%%%%%%%%%%%%%

%%%%%%%%%%%%%%%%%%%%%%%%%%

%\section{Examples }

\subsection{Integral condition insufficient when $\vt$ is not one-component}\label{subsec:counterex}

Let 
\begin{equation*}
\alpha_n = 2^{-n}, \; r_n = 1 - \frac{1}{n}2^{-2n}, \; z_n = r_n e^{i\alpha_n}, \ n=1,2,\ldots \ ,
\end{equation*}
and let $\vt$ be the Blaschke product
\begin{equation*}
\vt(z) = \prod_{n=1}^\infty \frac{-\overline{z}_n}{|z_n|} \frac{z-z_n}{1-\overline{z}_nz}, \quad z \in \DD.
\end{equation*}
The sequence $\{z_n\}$ is interpolating (for $H^\infty$), since it is  hyperbolically separated in $\D$ and $\sum_n (1-|z_n|)\delta_{z_n}$ is a Carleson measure for $H^2$.
Therefore the sequence of normalized reproducing kernels
\[
 \{  {(1-|z_n|)^{1/2}}k_{z_n}\}_n,   \   k_{z_n}=\frac1 {1-\bar{z}_nz}, 
\]
is a Riesz basis for  $K_\vt$. In particular any Hilbert-Schmidt operator $C: K_\vt \to H^2$ satisfies
\beq
\label{eq:HS}
\sum_n(1-|z_n|) \langle C k_{z_n},  C k_{z_n} \rangle  <\infty.
\eeq

Let $\Delta = \{z \, : \, |z-\frac12| < \frac12\}$ and let   $\vf(z) = (1+z)/2$, a conformal mapping of $\D$ onto 
$\Delta$. We will show that \eqref{eq:HS} with $C=C_\vf$ fails, while  
\begin{equation}
\label{eq:onecompcond}
\int_\DD \left(\frac{1-|\vt(z)|^2}{1-|z|^2}\right)^3 N_\vf(z) \, dA(z) < \infty.
\end{equation}

Indeed, since 
 $N_\vf(z) = -\log |2z-1| \gtrsim \dist(z, \Delta)$ for $z \in \Delta$, we have 
\begin{align*}
(1-|z_n|) \langle C_\vf k_{z_n}, C_\vf k_{z_n} \rangle_{H^2} &\gtrsim (1-|z_n|) \int_\Delta \left|z - \frac{1}{\overline{z_n}} \right|^{-4} N_\vf(z) \, dA(z) \\ &\gtrsim \frac{1-|z_n|}{\dist(z_n, \Delta)} \sim \frac{1}{n}.
\end{align*}
This yields divergence of the series on the right hand side of \eqref{eq:HS}. Therefore $C_\vf$ is not a Hilbert-Schmidt operator.

To prove \eqref{eq:onecompcond} we need an auxiliary statement.

\begin{lemma} \label{lem:uniformsum}
For any $a > 0$, the sum
\begin{equation*}
\sum_{n=1}^\infty \left( \frac{1-|z_n|}{|1-\overline{z_n}z|}\right)^a
\end{equation*}
is uniformly convergent in $\overline{\DD}$. In particular, it is uniformly bounded.
\end{lemma}

\begin{proof}
For $z \in \overline{\DD}$ and $0 < t < 1$, let $M(t,z) = \{n \geq 1 \, : \, \frac{1-|z_n|}{|1-\overline{z_n}z|} > t \}$ and let $N(t,z) = |M(t,z)|$ be the number of points in $M(t,z)$. Pick $n$ and $k$ such that $n > k \geq N(t,z)-1$ and $z_n, z_k \in M(t,z)$. Then
\begin{equation*}
2^{-k}|z| \lesssim |\overline{(z_n - z_k)}z| \lesssim \frac{1-|z_k|}{t} \leq \frac{2^{-2k}}{t},
\end{equation*}
from which $|z|t \lesssim 2^{-k} \lesssim 2^{-N(t,z)}$. That is, there exists a constant $C > 0$ such that
\begin{equation*}
N(t,z) \leq \log \frac{C}{|z|t}.
\end{equation*}
In combination with the identity
\begin{equation*}
\sum_{n=1}^\infty \left( \frac{1-|z_n|}{|1-\overline{z_n}z|}\right)^\alpha = \alpha \int_0^1 t^{\alpha-1} N(t,z) \, dt,
\end{equation*}
the statement of the lemma clearly follows.
\end{proof}

In order to prove \eqref{eq:onecompcond} we first  observe that  for
 $0 < \epsilon < 1$  
\begin{align*}
\frac{1-|\vt(z)|^2}{1-|z|^2} &\leq \sum_n \frac{1-|z_n|^2}{|1-\overline{z_n}z_n|^2} \\ &\leq \left( \sum_n \left(\frac{1-|z_n|^2}{|1-\overline{z_n}z|}\right)^{1-\epsilon} \right)^{2/3} \left( \sum_n \frac{(1-|z_n|^2)^{1+2\epsilon}}{|1-\overline{z_n}z|^{4+2\epsilon}}\right)^{1/3}.
\end{align*}
Let $\Delta' = \{z \, : \, |z-1/2| < 1/4\}$. Using Lemma \ref{lem:uniformsum} and the fact that $N_\vf(z) \sim \dist(z, \Delta)$ for $z \in \Delta \setminus \Delta'$ , we obtain 
\begin{multline*}
\int_{\DD\setminus\Delta'} \left(\frac{1-|\vt(z)|^2}{1-|z|^2}\right)^3 N_\vf(z) \, dA(z) \lesssim \\
 \sum_n (1-|z_n|^2)^{1+2\epsilon} \int_\Delta \left|z-\frac{1}{\overline{z_n}}\right|^{-4-2\epsilon} \dist(z,\Delta) \, dA(z)   \sim  \\  
\sum_n \left(\frac{1-|z_n|}{\dist(z_n, \Delta)}\right)^{1+2\epsilon} \sim \sum_n \frac{1}{n^{1+2\epsilon}} < \infty,
\end{multline*}
which of course implies \eqref{eq:onecompcond}.

\end{document}